\documentclass[11pt]{article}
\usepackage[margin=1.0in]{geometry}

\usepackage{verbatim}

\usepackage{amsmath}
\usepackage{amsthm}
\usepackage{amsfonts}
\usepackage{multicol}
\usepackage{subfigure}
\usepackage[numbers,sort&compress]{natbib}
\usepackage{booktabs}
\usepackage{titlesec}

\usepackage{enumitem}
\setlist[enumerate]{label={(\arabic*)}}

\usepackage{color}
\newcommand{\red}[1]{\textcolor{red}{#1}}

\usepackage{tikz}
\usetikzlibrary{calc}

\usepackage{tkz-graph}

\newtheorem{theorem}{Theorem}[section]
\newtheorem{lemma}[theorem]{Lemma}

\newtheorem{corollary}[theorem]{Corollary}
\newtheorem{fact}[theorem]{Fact}

\theoremstyle{definition}

\newtheorem{conjecture}[theorem]{Conjecture}
\newtheorem{problem}{Open Problem}

\newcommand{\squishlist}{
 \begin{list}{$\bullet$}
  { \setlength{\itemsep}{0pt}
     \setlength{\parsep}{3pt}
     \setlength{\topsep}{3pt}
     \setlength{\partopsep}{0pt}
     \setlength{\leftmargin}{2.5em}
     \setlength{\labelwidth}{1em}
     \setlength{\labelsep}{0.5em} } }

\newcommand{\squishlisttwo}{
 \begin{list}{$\triangleright$}
  { \setlength{\itemsep}{0pt}
     \setlength{\parsep}{0pt}
    \setlength{\topsep}{0pt}
    \setlength{\partopsep}{0pt}
    \setlength{\leftmargin}{2em}
    \setlength{\labelwidth}{1.5em}
    \setlength{\labelsep}{0.5em} } }

\newcommand{\squishend}{
  \end{list}  }

\begin{document}

\title{Dichotomizing $k$-vertex-critical $H$-free graphs for $H$ of order four}
\author{Ben Cameron\\
\and
Ch\'{i}nh T. Ho\`{a}ng\\
\and
Joe Sawada\\
}

\date{\today}

\maketitle

\begin{abstract}


For $k \geq 3$, we prove (i) there is a finite number of $k$-vertex-critical $(P_2+\ell P_1)$-free graphs and (ii)  $k$-vertex-critical $(P_3+P_1)$-free graphs have at most $2k-1$ vertices.  Together with previous research, these results imply the following characterization where $H$ is a graph of order four:  There is a finite number of $k$-vertex-critical $H$-free graphs for fixed $k \geq 5$ if and only if $H$ is one of  $\overline{K_4}, P_4, P_2 + 2P_1$, or  $P_3 + P_1$. Our results imply the existence of new polynomial-time certifying algorithms for deciding the $k$-colorability of $(P_2+\ell P_1)$-free graphs for fixed $k$.

\end{abstract}

\section{Introduction}

There are 11 non-isomorphic simple unlabeled graphs with 4 vertices as illustrated below:  
\begin{center}
\resizebox{5in}{!}{\includegraphics{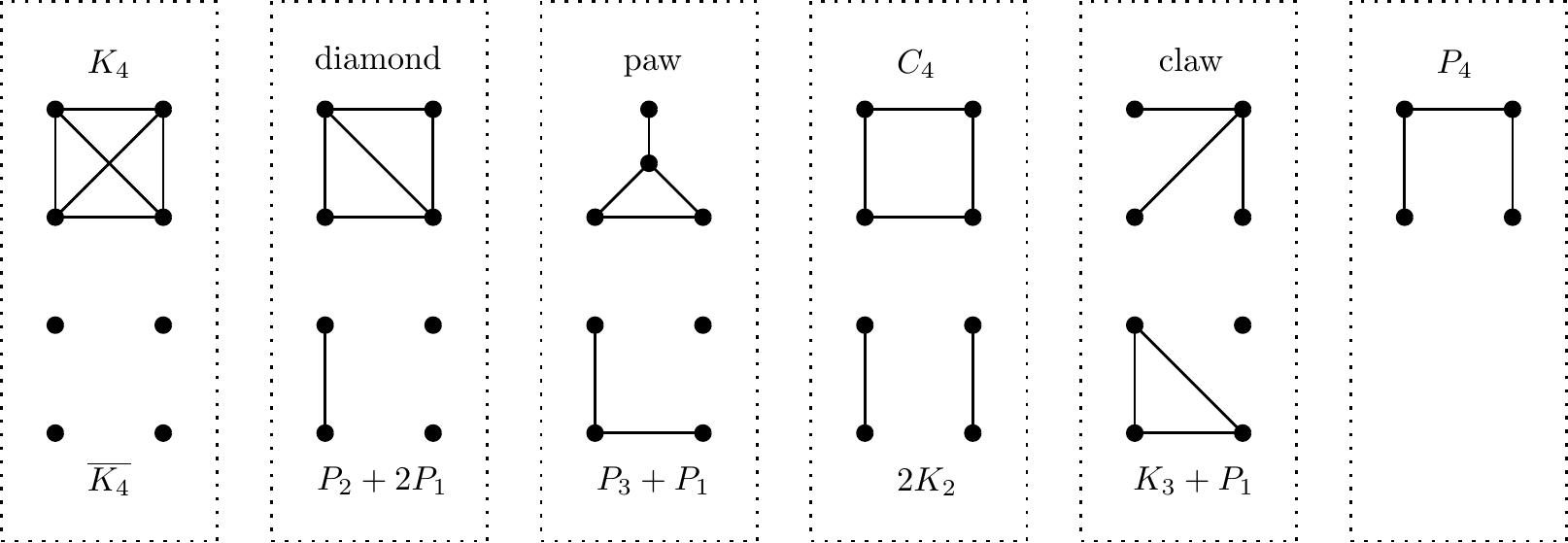}}
\end{center}
The graphs in each column are complements of each other; $P_4$ is self-complementary.  In this paper we characterize whether or
not there is a finite number of $k$-vertex-critical $H$-free graphs, where $k \geq 3$ and $H$ has four vertices.  The current state of the art, as discussed 
in Section~\ref{sec:results},  is illustrated in Table~\ref{tab:stateofart}.

\begin{table}[h]
\begin{center}
\begin{tabular} {c  |  c c c c c c c c c c c} \small
$k/H $  	& $K_4$ 	&  $\overline{K_4}$ 	&  diamond 	& $P_2 + 2P_1$ &  paw 	& $P_3 + P_1$ &  $2K_2$ & $C_4$ &  claw 	& $K_3 + P_1$  &  $P_4$ \\ \hline
3  		& $\infty$ 		& 3				& $\infty$ 			& 2 			  & $\infty$ 	& 2 			& 2 		 & $\infty$ 	& $\infty$ 		& $\infty$  	 	& 1 	    \\
4 		& $\infty$ 		& finite			& $\infty$ 			& 9			  & $\infty$ 	& 8 			& 7 		 & $\infty$ 	& $\infty$  	& $\infty$  	 	& 1 	    \\
5+ 		& $\infty$ 		& finite				& $\infty$ 		& \red{?}		  	& $\infty$ 		& finite		& $\infty$ 		 & $\infty$ 	& $\infty$  	& $\infty$  	 	& 1 	    \\
\end{tabular}
\end{center}
\vspace{-0.15in}
\caption{The number of $k$-vertex-critical $H$-free graphs}
\label{tab:stateofart}
\end{table}

To answer the open question in the table, we demonstrate for all $k \geq 3$ that  there is a finite number of $k$-vertex-critical  ($P_2 + \ell P_1$)-free graphs  for all $\ell \geq 1$. 
Additionally, we provide a simple proof that $k$-vertex-critical  ($P_3 + P_1$)-free graphs have at most $2k-1$ vertices.  
This leads to the following dichotomy theorem.

\begin{theorem} \label{thm:dichotomy}
Let $H$ be a graph containing at most four vertices.
There is a finite number of $k$-vertex-critical $H$-free graphs for fixed $k \geq 5$ if and only if $H$ is an induced subgraph of  $\overline{K_4}, P_4, P_2 + 2P_1$, or  $P_3 + P_1$.
\end{theorem}

While classifying vertex-critical graphs is of interest in its own right, there are also algorithmic consequences that arise from determining that a class of graphs has only finitely many $k$-vertex-critical graphs. An algorithm is \textit{certifying} if together with each output it also includes a simple and easily verifiable witness that the output is correct. For the problem of $k$-\textsc{Coloring}, that is, determining if a graph is $k$-colorable for fixed $k$, an algorithm that returns ``yes'' could also include a $k$-coloring as a certificate. If the algorithm returns ``no'' to this question, then a certificate could be an induced subgraph that is $(k+1)$-vertex-critical, since any graph that is not $k$-colorable must contain a $(k+1)$-vertex-critical induced subgraph. If we restrict the input graphs to a class of graphs that contains only finitely many $(k+1)$-vertex-critical graphs, then there is a polynomial-time algorithm to solve $k$-\textsc{Coloring} that can also return a no-certificate by searching for all $(k+1)$-vertex-critical graphs as induced subgraphs of the input graph (see \cite{P5banner2019} for more details). This is significant as $k$-\textsc{Coloring} is NP-complete in general \cite{Karp1972}. Moreover, $k$-\textsc{Coloring} $H$-free graphs remains NP-complete for all $k\ge 3$ if $H$ contains a cycle \cite{KaminskiLozin2007} or a claw \cite{Holyer1981, LevenGail1983}.

The remainder of this paper is presented as follows.
In Section~\ref{sec:back} we provide essential background definitions.
In Section~\ref{sec:results} we present the necessary background to prove our main result.  
In Section~\ref{sec:mainproof} we prove our dichotomizing Theorem~\ref{thm:dichotomy}. 
In Section~\ref{sec:copaw} we present results for $k$-vertex-critical ($P_3+P_1$)-free graphs. 
We conclude with future directions and open problems in Section~\ref{sec:openproblems}.

\subsection{Preliminary background, definitions, and notation}  \label{sec:back}

Let $G = (V(G),E(G))$ denote a simple undirected graph with vertex set $V(G)$ and edge set $E(G)$.  Let $\overline{G}$ denote the complement of $G$. The order of $G$ is number of vertices in $G$, i.e. $|V(G)|$.
Let $K_t$ denote the complete graph on $t \geq 1$ vertices, let  $C_t$ denote the induced cycle on $t \geq 3$ vertices, and let $P_t$ denote the induced path on $t\geq 1$ vertices.
For graphs $H_1,H_2,\ldots, H_j$, a graph is $(H_1,H_2,\ldots,H_j)$-free if it does not contain $H_i$ as an induced subgraph for any $i=1,2,\ldots,j$. If $j=1$, we simply write $H_1$-free.
The join of two graphs $G$ and $H$ is denoted $G\vee H$.  The disjoint union of two graphs is denoted $G+H$. For $S\subseteq V(G)$, let $G-S$ denote the graph obtained by removing all vertices in $S$ from $G$ along with all of their incident edges. If $S=\{v\}$, we simply write $G-v$ for $G-\{v\}$.

Let $\chi(G), \alpha(G)$,  and $\omega(G)$ denote the \textit{chromatic number} (the minimum number of colors required to color the vertices of $G$), the \textit{independence number} (the order of the largest independent set), and the \textit{clique number} (the order of the largest clique) of $G$, respectively.  Let $R(r,s)$, for $r,s \geq 1$,  denote the Ramsey numbers which have the following property:

\begin{theorem}[Ramsey's Theorem \cite{Ramsey}]
There exists a least positive integer, denoted $R(r,s)$, such that every graph $G$ with at least $R(r,s)$ vertices contains either a clique on $r$ vertices or an independent set on $s$ vertices.
\end{theorem}
\noindent
%
%

%

\subsection{Results on $k$-vertex-critical graphs}  \label{sec:results}

A graph $G$ is \textit{$k$-vertex-critical} if $\chi(G)=k$ but $\chi(G-v)<k$ for all $v\in V$.  
The only $k$-vertex-critical graphs for $k\le 2$ are the complete graphs, $K_1$ and $K_2$.    When $k=3$, the following well-known fact implies that there is an infinite number of 3-vertex-critical graphs. 
\begin{fact} \label{rem:3critical}
 $G$ is 3-vertex-critical if and only if $G$ is $C_{2n+1}$ for $n \geq 1$.
\end{fact}

\noindent
Based on this fact, note that $C_{2n+1}\vee K_m$ is $(3+m)$-vertex-critical for all $n\ge 1$.
Thus, there is an infinite number of $k$-vertex-critical graphs for all $k \geq 4$.
However, when we restrict the structure of the graphs in question by forbidding certain induced subgraphs, we can be left with finitely many $k$-vertex-critical graphs for some values of $k$. A simple example is to forbid $P_3$. In this case, we are left with the disjoint union of cliques, and hence the only $k$-vertex-critical $P_3$-free graph is $K_k$ for all $k\ge 3$. An old result of Seinche \cite{Seinche} showed that for a $P_4$-free graph $G$, we have $\chi(G) = \omega(G)$. Thus, the following fact holds.

%
\begin{fact}\label{fact:P4freeperfect}
The only $k$-vertex critical $P_4$-free graph is $K_k$, for $k \geq 1$.
\end{fact}
\noindent By applying Ramsey's Theorem, we get the following.
\begin{fact}\label{fact:stable}
There is a finite number of $k$-vertex-critical $\overline{K_k}$-free graphs, where $k \geq 1$.
\end{fact}

\noindent
Forbidding small graphs, however, does not always lead to a finite number of $k$-vertex-critical graphs.  For instance, the following three results are known.  

\begin{fact}[\cite{Hoang2015}]\label{fact:infif2K2}
There is an infinite number of $k$-vertex-critical $2K_2$-free graphs, where $k \geq 5$.
\end{fact}

\noindent A classical result of Erd\H{o}s \cite{Erdos} that for all $k$, there exist $k$-chromatic graphs of arbitrarily large girth gives the next fact.

\begin{fact}\label{fact:infifcycle}
If $H$ contains an induced cycle, then there is an infinite number of $k$-vertex-critical $H$-free graphs, where $k \geq 3$.
\end{fact}

\begin{fact}\label{fact:infifclaw}
There is an infinite number of $k$-vertex-critical claw-free graphs, where $k \geq 3$.
\end{fact}
\noindent An example of an infinite family of $k$-vertex-critical claw-free graphs is the family of graphs obtained by taking an odd induced cycle with vertices labeled $1,2,\ldots, 2t+1$ ($t \geq 2$) and substituting a clique of order $k-2$ for each vertex with an even label.


The following result is known from concurrent research.  In Section~\ref{sec:copaw}, we provide a more concise proof of this result including a tight upper bound of $2k-1$ for the number of vertices in
$k$-vertex-critical ($P_3+P_1$)-free graph.

\begin{fact}[\cite{Cameron2020}]  \label{fact:P3P1} 
There is a finite number of $k$-vertex-critical ($P_3+P_1$)-free graphs, where $k \geq 3$.
\end{fact}

\noindent  
When forbidding two induced subgraphs, there are more positive results with respect to finiteness. 
For instance,  there are exactly $13$ $5$-vertex-critical $(P_5,C_5)$-free graphs~ \cite{Hoang2015}.
More generally, there is a finite number of $k$-vertex-critical $(P_5,\overline{P_5})$-free graphs for all $k \geq 5$ as well as a structural characterization for these graphs \cite{Dhaliwal2017}.
It is also known that there are only finitely many $k$-vertex-critical $(P_t,K_{s,s})$-free graphs for all $k$, extending  an analogous result for $(P_6,C_4)$-free graphs~\cite{Hell2017} . 
There is a finite number of $4$- and $5$-vertex-critical $(P_6,banner)$-free graphs where $banner$ is the graph obtained from a $C_4$ by attaching a leaf to one of its vertices \cite{Huang2019}. 
As well, there are only finitely many $6$-vertex-critical $(P_5,banner)$-free graphs \cite{P5banner2019}.   
In recent work, it was shown for $k \geq 5$ that there is a finite number of $k$-vertex-critical $(P_5,H)$-free graphs, where $H$ has four vertices and $H$ is neither $2K_2$ nor $K_3+P_1$ ~\cite{Cameron2020}. 
One of the more significant recent result is the following dichotomy theorem:
%
\begin{theorem}[\cite{Chud4critical2020}]\label{thm:finite4critical}
Let $H$ be a graph. There is a finite number of $4$-vertex-critical $H$-free  graphs if and only if $H$ is an induced subgraph of $P_6$, $2P_3$, or $P_4+\ell P_1$ for some $\ell \in\mathbb{N}$.
\end{theorem}

\noindent
Applying this result to graphs $H$ on four vertices we obtain the following corollary.

\begin{corollary}
Let $H$ be a graph containing four vertices.
There is a finite number of $4$-vertex-critical $H$-free graphs if and only if $H$ is one of $\overline{K_4}, P_4,  P_2 + 2P_1, P_3 + P_1$, or $2K_2$.
\end{corollary}

Except for $\overline{K_4}$, we can determine the number of $4$-vertex-critical $H$-free graphs by filtering from the list of $80$ $4$-vertex-critical $P_6$-free graphs provided in \cite{Chudnovsky2020}.  In particular, there are nine 4-vertex-critical ($P_2 + 2P_1$)-free graphs, $G_1,G_2,\ldots, G_9$ in Figure~\ref{fig:4critP22P1free};  there are eight 4-vertex-critical ($P_3 + P_1$)-free graphs, $G_1,G_2,\ldots, G_8$ in Figure~\ref{fig:4critP22P1free};  there are seven 4-vertex-critical ($2K_2$)-free graphs, all those from Figure~\ref{fig:4critP22P1free} except for $G_4,G_5,G_6$ and $G_8$.  
Using the same list of 80 graphs there are at least 50 4-vertex-critical $\overline{K_4}$-free graphs and the only  4-vertex-critical $P_4$-free graph is $K_4$. 


\setcounter{subfigure}{0}
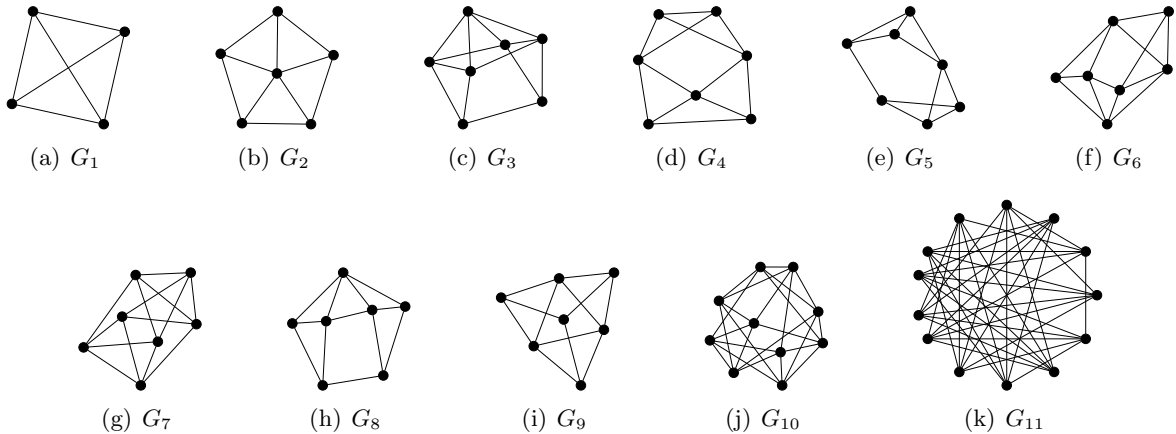
\begin{figure}[!h]
\def\c{0.3}
\def\r{1}
\centering
\subfigure[$G_1$]{
\scalebox{\c}{
\begin{tikzpicture}
\GraphInit[vstyle=Classic]
\Vertex[L=\hbox{},x=5.0cm,y=4.0987cm]{v0}
\Vertex[L=\hbox{},x=4.0675cm,y=0.0cm]{v1}
\Vertex[L=\hbox{},x=0.9343cm,y=5.0cm]{v2}
\Vertex[L=\hbox{},x=0.0cm,y=0.8987cm]{v3}
\Edge[](v0)(v1)
\Edge[](v0)(v2)
\Edge[](v0)(v3)
\Edge[](v1)(v2)
\Edge[](v1)(v3)
\Edge[](v2)(v3)
\end{tikzpicture}}}
\qquad
\subfigure[$G_2$]{
\scalebox{\c}{
\begin{tikzpicture}
\GraphInit[vstyle=Classic]
\Vertex[L=\hbox{},x=0.0cm,y=3.1235cm]{v0}
\Vertex[L=\hbox{},x=5.0cm,y=3.0618cm]{v1}
\Vertex[L=\hbox{},x=0.942cm,y=0.0422cm]{v2}
\Vertex[L=\hbox{},x=2.5253cm,y=5.0cm]{v3}
\Vertex[L=\hbox{},x=4.0006cm,y=0.0cm]{v4}
\Vertex[L=\hbox{},x=2.4923cm,y=2.2487cm]{v5}
\Edge[](v0)(v2)
\Edge[](v0)(v3)
\Edge[](v0)(v5)
\Edge[](v1)(v3)
\Edge[](v1)(v4)
\Edge[](v1)(v5)
\Edge[](v2)(v4)
\Edge[](v2)(v5)
\Edge[](v3)(v5)
\Edge[](v4)(v5)
\end{tikzpicture}}}
\qquad
\subfigure[$G_3$]{
\scalebox{\c}{
\begin{tikzpicture}
\GraphInit[vstyle=Classic]
\Vertex[L=\hbox{},x=1.4829cm,y=0.0cm]{v0}
\Vertex[L=\hbox{},x=5.0cm,y=3.784cm]{v1}
\Vertex[L=\hbox{},x=1.8219cm,y=2.3457cm]{v2}
\Vertex[L=\hbox{},x=3.3514cm,y=3.5216cm]{v3}
\Vertex[L=\hbox{},x=0.0cm,y=2.7571cm]{v4}
\Vertex[L=\hbox{},x=4.9936cm,y=1.0046cm]{v5}
\Vertex[L=\hbox{},x=1.7076cm,y=5.0cm]{v6}
\Edge[](v0)(v2)
\Edge[](v0)(v4)
\Edge[](v0)(v5)
\Edge[](v1)(v2)
\Edge[](v1)(v3)
\Edge[](v1)(v5)
\Edge[](v1)(v6)
\Edge[](v2)(v4)
\Edge[](v2)(v6)
\Edge[](v3)(v4)
\Edge[](v3)(v5)
\Edge[](v3)(v6)
\Edge[](v4)(v6)
\end{tikzpicture}}}
\qquad
\subfigure[$G_4$]{
\scalebox{\c}{
\begin{tikzpicture}
\GraphInit[vstyle=Classic]
\Vertex[L=\hbox{},x=0.0cm,y=2.8486cm]{v0}
\Vertex[L=\hbox{},x=2.5522cm,y=1.2761cm]{v1}
\Vertex[L=\hbox{},x=0.9047cm,y=4.8687cm]{v2}
\Vertex[L=\hbox{},x=5.0cm,y=0.2269cm]{v3}
\Vertex[L=\hbox{},x=3.4578cm,y=5.0cm]{v4}
\Vertex[L=\hbox{},x=0.4536cm,y=0.0cm]{v5}
\Vertex[L=\hbox{},x=4.8105cm,y=3.0363cm]{v6}
\Edge[](v0)(v1)
\Edge[](v0)(v2)
\Edge[](v0)(v4)
\Edge[](v0)(v5)
\Edge[](v1)(v3)
\Edge[](v1)(v5)
\Edge[](v1)(v6)
\Edge[](v2)(v4)
\Edge[](v2)(v6)
\Edge[](v3)(v5)
\Edge[](v3)(v6)
\Edge[](v4)(v6)
\end{tikzpicture}}}
\qquad
\subfigure[$G_5$]{
\scalebox{\c}{
\begin{tikzpicture}
\GraphInit[vstyle=Classic]
\Vertex[L=\hbox{},x=0.0cm,y=3.5735cm]{v0}
\Vertex[L=\hbox{},x=5.0cm,y=0.7618cm]{v1}
\Vertex[L=\hbox{},x=2.7975cm,y=5.0cm]{v2}
\Vertex[L=\hbox{},x=3.5566cm,y=0.0cm]{v3}
\Vertex[L=\hbox{},x=2.1067cm,y=3.9882cm]{v4}
\Vertex[L=\hbox{},x=1.5393cm,y=1.0533cm]{v5}
\Vertex[L=\hbox{},x=4.229cm,y=2.6376cm]{v6}
\Edge[](v0)(v2)
\Edge[](v0)(v4)
\Edge[](v0)(v5)
\Edge[](v1)(v3)
\Edge[](v1)(v5)
\Edge[](v1)(v6)
\Edge[](v2)(v4)
\Edge[](v2)(v6)
\Edge[](v3)(v5)
\Edge[](v3)(v6)
\Edge[](v4)(v6)
\end{tikzpicture}}}
\qquad
\subfigure[$G_6$]{
\scalebox{\c}{
\begin{tikzpicture}
\GraphInit[vstyle=Classic]
\Vertex[L=\hbox{},x=2.5386cm,y=4.5693cm]{v0}
\Vertex[L=\hbox{},x=1.4159cm,y=2.1384cm]{v1}
\Vertex[L=\hbox{},x=4.9459cm,y=2.4322cm]{v2}
\Vertex[L=\hbox{},x=2.2749cm,y=0.0cm]{v3}
\Vertex[L=\hbox{},x=5.0cm,y=5.0cm]{v4}
\Vertex[L=\hbox{},x=0.0cm,y=2.0599cm]{v5}
\Vertex[L=\hbox{},x=2.8306cm,y=1.5095cm]{v6}
\Edge[](v0)(v1)
\Edge[](v0)(v2)
\Edge[](v0)(v4)
\Edge[](v0)(v5)
\Edge[](v1)(v3)
\Edge[](v1)(v5)
\Edge[](v1)(v6)
\Edge[](v2)(v3)
\Edge[](v2)(v4)
\Edge[](v2)(v6)
\Edge[](v3)(v5)
\Edge[](v3)(v6)
\Edge[](v4)(v6)
\end{tikzpicture}}}
\qquad
\subfigure[$G_7$]{
\scalebox{\c}{
\begin{tikzpicture}
\GraphInit[vstyle=Classic]
\Vertex[L=\hbox{},x=0.0cm,y=1.685cm]{v0}
\Vertex[L=\hbox{},x=2.308cm,y=4.8945cm]{v1}
\Vertex[L=\hbox{},x=1.7094cm,y=3.0494cm]{v2}
\Vertex[L=\hbox{},x=4.7467cm,y=5.0cm]{v3}
\Vertex[L=\hbox{},x=2.5359cm,y=0.0cm]{v4}
\Vertex[L=\hbox{},x=3.2912cm,y=1.9375cm]{v5}
\Vertex[L=\hbox{},x=5.0cm,y=2.7145cm]{v6}
\Edge[](v0)(v1)
\Edge[](v0)(v2)
\Edge[](v0)(v4)
\Edge[](v0)(v5)
\Edge[](v1)(v3)
\Edge[](v1)(v5)
\Edge[](v1)(v6)
\Edge[](v2)(v3)
\Edge[](v2)(v4)
\Edge[](v2)(v6)
\Edge[](v3)(v5)
\Edge[](v3)(v6)
\Edge[](v4)(v5)
\Edge[](v4)(v6)
\end{tikzpicture}}}
\qquad
\subfigure[$G_8$]{
\scalebox{\c}{
\begin{tikzpicture}
\GraphInit[vstyle=Classic]
\Vertex[L=\hbox{},x=4.0302cm,y=0.4387cm]{v0}
\Vertex[L=\hbox{},x=1.4967cm,y=2.8477cm]{v1}
\Vertex[L=\hbox{},x=3.5366cm,y=3.3444cm]{v2}
\Vertex[L=\hbox{},x=0.0cm,y=2.7474cm]{v3}
\Vertex[L=\hbox{},x=5.0cm,y=3.4962cm]{v4}
\Vertex[L=\hbox{},x=1.3457cm,y=0.0cm]{v5}
\Vertex[L=\hbox{},x=2.2522cm,y=5.0cm]{v6}
\Edge[](v0)(v2)
\Edge[](v0)(v4)
\Edge[](v0)(v5)
\Edge[](v1)(v2)
\Edge[](v1)(v3)
\Edge[](v1)(v5)
\Edge[](v1)(v6)
\Edge[](v2)(v4)
\Edge[](v2)(v6)
\Edge[](v3)(v5)
\Edge[](v3)(v6)
\Edge[](v4)(v6)
\end{tikzpicture}}}
\qquad
\subfigure[$G_9$]{
\scalebox{\c}{
\begin{tikzpicture}
\GraphInit[vstyle=Classic]
\Vertex[L=\hbox{},x=3.5345cm,y=0.0cm]{v0}
\Vertex[L=\hbox{},x=0.0cm,y=3.8918cm]{v1}
\Vertex[L=\hbox{},x=5.0cm,y=5.0cm]{v2}
\Vertex[L=\hbox{},x=1.4434cm,y=1.7413cm]{v3}
\Vertex[L=\hbox{},x=4.5534cm,y=2.4649cm]{v4}
\Vertex[L=\hbox{},x=2.7724cm,y=2.9206cm]{v5}
\Vertex[L=\hbox{},x=2.5682cm,y=4.7449cm]{v6}
\Edge[](v0)(v3)
\Edge[](v0)(v4)
\Edge[](v0)(v5)
\Edge[](v1)(v3)
\Edge[](v1)(v5)
\Edge[](v1)(v6)
\Edge[](v2)(v4)
\Edge[](v2)(v5)
\Edge[](v2)(v6)
\Edge[](v3)(v4)
\Edge[](v3)(v6)
\Edge[](v4)(v6)
\end{tikzpicture}}}
\qquad
\subfigure[$G_{10}$]{
\scalebox{\c}{
\begin{tikzpicture}
\GraphInit[vstyle=Classic]
\Vertex[L=\hbox{},x=0.4cm,y=3.5cm]{v0}
\Vertex[L=\hbox{},x=2.2404cm,y=5cm]{v1}
\Vertex[L=\hbox{},x=3.2cm,y=-0.25cm]{v2}
\Vertex[L=\hbox{},x=1.9519cm,y=2.5023cm]{v3}
\Vertex[L=\hbox{},x=0.0cm,y=1.7474cm]{v4}
\Vertex[L=\hbox{},x=5.0cm,y=1.6249cm]{v5}
\Vertex[L=\hbox{},x=3.6901cm,y=5.0cm]{v6}
\Vertex[L=\hbox{},x=4.8cm,y=3.018cm]{v7}
\Vertex[L=\hbox{},x=3.1303cm,y=1.223cm]{v8}
\Vertex[L=\hbox{},x=1.0528cm,y=0.3089cm]{v9}
\Edge[](v0)(v1)
\Edge[](v0)(v2)
\Edge[](v0)(v3)
\Edge[](v0)(v6)
\Edge[](v0)(v9)
\Edge[](v1)(v4)
\Edge[](v1)(v5)
\Edge[](v1)(v6)
\Edge[](v1)(v7)
\Edge[](v2)(v4)
\Edge[](v2)(v5)
\Edge[](v2)(v7)
\Edge[](v2)(v8)
\Edge[](v3)(v4)
\Edge[](v3)(v5)
\Edge[](v3)(v6)
\Edge[](v3)(v9)
\Edge[](v4)(v8)
\Edge[](v4)(v9)
\Edge[](v5)(v7)
\Edge[](v5)(v8)
\Edge[](v6)(v7)
\Edge[](v6)(v8)
\Edge[](v7)(v9)
\Edge[](v8)(v9)
\end{tikzpicture}}}
\qquad
\subfigure[$G_{11}$]{
\scalebox{\c}{
\begin{tikzpicture}
\GraphInit[vstyle=Classic]
\Vertex[L=\hbox{},x=4cm,y=8cm]{v0}
\Vertex[L=\hbox{},x=6.1cm,y=7.4044cm]{v12}
\Vertex[L=\hbox{},x=7.5cm,y=5.9365cm]{v11}
\Vertex[L=\hbox{},x=8cm,y=4cm]{v10}
\Vertex[L=\hbox{},x=7.5cm,y=2.0635cm]{v9}
\Vertex[L=\hbox{},x=6.1cm,y=0.5956cm]{v8}
\Vertex[L=\hbox{},x=4cm,y=0cm]{v7}
\Vertex[L=\hbox{},x=1.9cm,y=0.5956cm]{v6}
\Vertex[L=\hbox{},x=0.5cm,y=2.0635cm]{v5}
\Vertex[L=\hbox{},x=0.1cm,y=3.1112cm]{v4}
\Vertex[L=\hbox{},x=0.1cm,y=4.8888cm]{v3}
\Vertex[L=\hbox{},x=0.5cm,y=5.9365cm]{v2}
\Vertex[L=\hbox{},x=1.9cm,y=7.4044cm]{v1}
\Edge[](v0)(v4)
\Edge[](v0)(v6)
\Edge[](v0)(v7)
\Edge[](v0)(v8)
\Edge[](v0)(v10)
\Edge[](v0)(v11)
\Edge[](v1)(v4)
\Edge[](v1)(v5)
\Edge[](v1)(v6)
\Edge[](v1)(v7)
\Edge[](v1)(v8)
\Edge[](v1)(v11)
\Edge[](v2)(v6)
\Edge[](v2)(v7)
\Edge[](v2)(v8)
\Edge[](v2)(v10)
\Edge[](v2)(v11)
\Edge[](v2)(v12)
\Edge[](v3)(v7)
\Edge[](v3)(v8)
\Edge[](v3)(v9)
\Edge[](v3)(v10)
\Edge[](v3)(v11)
\Edge[](v3)(v12)
\Edge[](v4)(v8)
\Edge[](v4)(v9)
\Edge[](v4)(v10)
\Edge[](v4)(v12)
\Edge[](v5)(v8)
\Edge[](v5)(v9)
\Edge[](v5)(v10)
\Edge[](v5)(v11)
\Edge[](v5)(v12)
\Edge[](v6)(v9)
\Edge[](v6)(v10)
\Edge[](v6)(v12)
\Edge[](v7)(v9)
\Edge[](v7)(v12)
\Edge[](v9)(v11)
\end{tikzpicture}}}
\caption{$4$-vertex-critical graphs.}%
\label{fig:4critP22P1free}%
\end{figure}


\section{Proof of Theorem \ref{thm:dichotomy}}\label{sec:mainproof}

In this section we prove Theorem~\ref{thm:dichotomy}.   Let $k \geq 5$.   
\squishlisttwo
\item From Fact~\ref{fact:P4freeperfect}, there is exactly one $k$-vertex-critical $P_4$-free graph.
\item From Fact~\ref{fact:stable}, there is a finite number of $k$-vertex-critical $\overline{K_4}$-free graphs. 
\item From Fact~\ref{fact:P3P1}, there is a finite number of $k$-vertex-critical $(P_3+P_1)$-free graphs. 
\item From Fact~\ref{fact:infif2K2}, there is an infinite number of $k$-vertex-critical $2K_2$-free graphs.  
\item From Fact~\ref{fact:infifclaw}, there is an infinite number of $k$-vertex-critical claw-free graphs.
\item From Fact~\ref{fact:infifcycle},  there is an infinite number of $k$-vertex-critical $H$-free graphs if $H$ is either $K_4$, diamond, paw, $C_4$,  or $K_3+P_1$.  
\squishend
Thus, to complete the proof, we need only consider the case when $H$ is  $P_2 + 2P_1$.   This case is handled by the following more general theorem that we prove in the following subsection.


\begin{theorem}\label{thm:finitelymanykcritP22P1free}
For all $k\ge 1$ and $\ell\ge 1$, there is a finite number of $k$-vertex-critical $(P_2+\ell P_1)$-free graphs.
\end{theorem}

\noindent
Based on the above theorem, there is a finite number of $k$-vertex-critical $(P_2+ 2P_1)$-free graphs. This completes the proof of Theorem~\ref{thm:dichotomy}.


\subsection{$(P_2+\ell P_1)$-free graphs}\label{sec:P2ellP1}

If there is a finite number of $k$-vertex-critical $H$-free graphs for a given $k$ and $H$, let $f(k,H)$ denote the maximal number of vertices in such a graph; otherwise, let $f(k,H) = \infty$.  Before proving 
Theorem~\ref{thm:finitelymanykcritP22P1free}, we prove two preliminary results.

\begin{lemma}\label{lem:nonneighbors}
If $G$ is a $(P_2+\ell P_1)$-free graph, $S\subseteq V(G)$ a maximal independent set, and $|S|\ge \ell+1$, then every vertex in $G-S$ has at most $\ell-1$ nonneighbors in $S$.
\end{lemma}
\begin{proof}
Consider a vertex $x \in G-S$. Since $S$ is maximal, $x$ has a neighbor in $S$, say $y$. Suppose $x\in G-S$ has $\ell$ nonneighbors in $S$, say $s_1,s_2,\ldots, s_{\ell}$.  Now, $\{x,y,s_1,s_2,\ldots, s_{\ell}\}$ induces a $P_2+\ell P_1$ in $G$, a contradiction. 
\end{proof}

\begin{lemma}\label{lem:3crit}
$f(3,P_2+\ell P_1)=2\ell+1$ for all $\ell\ge 0$.
\end{lemma}
\begin{proof}
Let $G$ be a $3$-vertex-critical $(P_2+\ell P_1)$-free graph. 
From Fact~\ref{rem:3critical}, $G=C_{2m+1}$ for some $m\ge 2$. Moreover, $C_{2m+1}$ is $(P_2+\ell P_1)$-free if and only if $m\le \ell$. Therefore, $G$ belongs to the set $\{C_{2m+1}:m=1,2,\ldots,\ell\}$. Hence, $f(3,P_2+\ell P_1)=2\ell+1$.
\end{proof}


\begin{proof}[Proof of Theorem \ref{thm:finitelymanykcritP22P1free}]
The proof is by induction of $k$. It is clear that $f(1,P_2+\ell P_1)=1$, $f(2,P_2+\ell P_1)=2$ and the case $k=3$ was shown in Lemma~\ref{lem:3crit}. Now let $G$ be a $k$-vertex-critical $(P_2+\ell P_1)$-free graph for some $k\ge 4$. Note that since $G$ is $k$-vertex-critical, $\omega(G)\le k$. If $\alpha(G)\le \ell$, then $|V(G)|< R(k+1,\ell+1)$, which is finite by Ramsey's Theorem \cite{Ramsey}. Therefore we can suppose that $\alpha(G)\ge \ell+1$. Let $S$ be a maximum independent set of $G$.  If $G-S$ is $(k-2)$-colorable, then we can extend this coloring to a $(k-1)$-coloring of $G$ by coloring all vertices in $S$ the same color, a contradiction. Therefore, $\chi(G-S)=k-1$. Hence, $G-S$ must have an induced $(k-1)$-vertex-critical subgraph, say $G'$. Therefore, $|V(G')|\le f(k-1,P_2+\ell P_1)$ by the induction hypothesis. Now if $|S|>(\ell-1)\cdot f(k-1,P_2+\ell P_1)$, then, by Lemma~\ref{lem:nonneighbors} and the Pigeonhole Principle, there is some vertex $s\in S$ that is adjacent to all vertices in $G'$. Therefore, the graph induced by $V(G')\cup \{s\}$ is a proper subgraph of $G$ and is $k$-vertex-critical, a contradiction. Therefore, $|S|\le(\ell-1)\cdot f(k-1,P_2+\ell P_1)$, and $|V(G)|<R(k+1,(\ell-1)f(k-1,P_2+\ell P_1)+1)$, which is finite, again by Ramsey's Theorem.
\end{proof}


\section{$(P_3+P_1)$-free graphs}\label{sec:copaw}

In this section we prove  the following theorem which improves on Fact~\ref{fact:P3P1} given in~\cite{Cameron2020}.  In particular, we demonstrate that any 
$k$-vertex-critical $(P_3+P_1)$-free graph has at most $2k-1$ vertices.  A graph obtaining this bound is $\overline{C_{2k-1}}$.

\begin{theorem}\label{thm:finitekcritP3P1free}
Let $k \geq 1$. If $G$ is a $k$-vertex-critical $(P_3+P_1)$-free graph, then $\alpha(G)\le 2$. Moreover, the maximum number of vertices in a $k$-vertex-critical $(P_3+P_1)$-free graph is $2k-1$.
\end{theorem}

\noindent
To prove this theorem, we apply the following known results.

\begin{theorem}[\cite{Olariu1988}]
A graph $G$ is $paw$-free if and only if every component of $G$ is triangle-free or complete multipartite.
\end{theorem}

\noindent
Since $paw=\overline{P_3+P_1}$, we get the following corollary immediately.

\begin{corollary}\label{cor:copawfreestructure}
A graph $G$ is $(P_3+P_1)$-free if and only if $G=H_1\vee H_2\vee \cdots \vee H_n$ where $\alpha(H_i)\le 2$ or $H_i$ is the disjoint union of cliques for all $i=1,2,\ldots, n$.
\end{corollary}
We will need the following two results.

\begin{lemma}[\cite{Dhaliwal2017}]\label{lem:critjoin}
A graph $G\vee H$ with $|V(G)|,|V(H)|\ge 1$ is $k$-vertex-critical if and only if $G$ is $k_1$-vertex-critical and $H$ is  $k_2$-vertex-critical such that $k_1+k_2=k$.
\end{lemma}

\begin{theorem}[\cite{Stehlik2003}]\label{thm:critconncomp}
If $G$ is $k$-vertex-critical and $\overline{G}$ is connected, then $G-v$ has a $(k-1)$-coloring in which every color class contains at least two vertices, for all $v \in V(G)$. 
\end{theorem}

An immediate consequence of this theorem is that a $k$-vertex-critical graph that is not the result of a graph join has at least $2k-1$ vertices.


\begin{proof}[Proof of Theorem~\ref{thm:finitekcritP3P1free}]
If $G=F\vee H$ is a $k$-vertex-critical $(P_3+P_1)$-free graph, then by Lemma~\ref{lem:critjoin} and induction, $|V(F)|\le 2k_1-1$ and $|V(H)|\le 2k_2-1$ where $k_1+k_2=k$, $\alpha(F)\le 2$ and $\alpha(H)\le 2$. Therefore, $|V(G)|\le 2k-2$ and $\alpha(G)\le 2$. So the result follows by induction. Otherwise, from Corollary~\ref{cor:copawfreestructure} and the fact that $G$ is vertex-critical and therefore connected, it follows that $\alpha(G)=2$. Moreover, an elementary lower bound on $\chi(G)$ is $\frac{|V(G)|}{\alpha(G)}$ and since $\chi(G)=k$ and $\alpha(G)=2$, we have $|V(G)|\le 2k$. 
If $|V(G)|=2k$, then let $v\in V(G)$ and consider $G-v$. Since $|V(G-v)|=2k-1$, every $(k-1)$-coloring of $G-v$ will have at least three vertices colored with the same color. But this means $G-v$ and therefore $G$ has an independent set of order at least three, contradicting $\alpha(G)=2$. Therefore, $|V(G)|\le 2k-1$. Also, from Theorem~\ref{thm:critconncomp}, $|V(G)|\ge 2k-1$, so if such a $G$ exists it must have order equal to $2k-1$. Finally, $\overline{C_{2k-1}}$ is a $k$-vertex-critical $(P_3+P_1)$-free graph of order $2k-1$ for all $k\ge 3$. 
\end{proof}

Earlier, we demonstrated that there are eight $4$-vertex-critical $(P_3+P_1)$-free graphs.  For $k>4$, the  $k$-vertex-critical $(P_3+P_1)$-free graphs were not previously known.
However, using our results to restrict the structure and order of $k$-vertex-critical $(P_3+P_1)$-free graphs we were able to employ \texttt{geng} in \texttt{nauty} \cite{McKay2014} to generate all graphs with order at most $11$ and independence number two (complements of those that are triangle-free) and then run an exhaustive computer search on the resulting graphs to find the exact number of $k$-vertex-critical $(P_3+P_1)$-free graphs for $k=5$ and $k=6$. Our findings are summarized in Table~\ref{tab:5and6crit} and the edge sets of all such graphs for $k=5$ are available in the Appendix.

\begin{table}[!h]
\begin{center}  \small
\renewcommand\arraystretch{1.2}
\begin{tabular}{|r|r|r|r|} 
\hline
$n$   &  $4$-vertex-critical  & $5$-vertex-critical & $6$-vertex-critical  \\ \hline
4   & $1$  & $0$			&  $0$  \\ \hline
5   & $0$  & $1$   				&  $0$  \\ \hline
6   & $1$  & $0$   				&  $1$   \\ \hline
7   & $6$  & $1$   				&  $0$    \\ \hline
8   & $0$  & $6$   				&  $1$    \\ \hline
9   & $0$  & $170$ 				&  $6$      \\ \hline
10  & $0$  & $0$   				&  $171$        \\ \hline
11  & $0$  & $0$   				&  $17828$   \\ \hline \hline
total &  $8$  & $178$ 			&  $18007$          \\ \hline
\end{tabular}
\caption{The number of $k$-vertex-critical $(P_3+P_1)$-free graphs of each order $n$ for $k=4,5,6$.}\label{tab:5and6crit}
\end{center}
\end{table}


As can be seen in Table~\ref{tab:5and6crit}, the number of $6$-vertex-critical $(P_3+P_1)$-free graphs is quite large, although finite. In practice, implementing a certifying algorithm for $k$-\textsc{Coloring} $(P_3+P_1)$-free graphs would require a complete list of all $(k+1)$-vertex-critical graphs. This is now possible for $k\le 5$. For $k>5$, though we do not have complete lists, we have been able to impose some structure on these graphs and we can completely describe those that are the result of graph joins in terms of $m$-vertex-critical graphs for $m<k+1$. However, when the graph has a connected complement, all we know is that it has to have independence number two and order $2(k+1)-1$.
It would be interesting to determine a structural characterization of all $k$-vertex-critical $(P_3+P_1)$-free graphs with connected complements, i.e. those with order $2k-1$.

\section{Open Problems}\label{sec:openproblems}

In this paper we characterize whether or not there is a finite number of $k$-vertex-critical $H$-free graphs for $k \geq 5$ and $H$ a graph on four vertices.
A natural extension to this research is to consider all 34 non-isomorphic graphs $H$ on five vertices.   
Of these 34, only four are cycle-free, $2K_2$-free, and claw-free, namely:  $\overline{K_5}, P_2 + 3P_1, P_3 + 2P_1$, and  $P_4+P_1$.  Thus, based on results in Section~\ref{sec:results} 
and Theorem~\ref{thm:finitelymanykcritP22P1free}, the only undetermined graphs $H$ are $P_3+2P_1$ and $P_4 + P_1$.   This leads to the following two problems.

\medskip

\noindent
{\bf Problem 1.} 
For which values of $k\geq 5$ is there a finite number of $k$-vertex-critical $P_3+2P_1$-free graphs?  \medskip

\noindent
{\bf Problem 2.}
For which values of $k\geq 5$ is there a finite number of $k$-vertex-critical $P_4+ P_1$-free graphs? \medskip

\noindent
Answering these two problems would lead to a dichotomy result for graphs on five vertices. More generally, for  graphs with $n \geq 5$ vertices,  a similar analysis shows that the only undetermined graphs
are $P_3 + (n{-}3)P_1$ and $P_4 + (n{-}4)P_1$.  This leads to the following generalizations of the above problems. \medskip

\noindent
{\bf Generalized Problem 1.} 
For which values of $k\geq 5$ and $r \geq 2$ is there a finite number of $k$-vertex-critical $P_3+rP_1$-free graphs? \medskip

\noindent
{\bf Generalized Problem 2.} 
For which values of $k\geq 5$ and $s \geq 1$ is there a finite number of $k$-vertex-critical $P_4+ sP_1$-free graphs? \medskip

\noindent
We conclude by making the following conjecture, which if proved correct would answer the above open problems.



\begin{conjecture}
Let $k\geq 5$. There is there a finite number of $k$-vertex-critical $H$-free graphs if and only if $H$ is 
an induced subgraph of $P_4+ sP_1$ for some $s \geq 0$.
\end{conjecture}

\bibliographystyle{plain}
\bibliography{kcritical}

\newpage
\newgeometry{left=25mm,right=25mm,top=20mm,bottom=20mm}

\section*{Appendix:  $5$-vertex-critical $(P_3+P_1)$-free graphs}
Below are the edge sets for all 178 $5$-vertex-critical $(P_3+P_1)$-free graphs. Vertices are labelled $0$ to $n-1$ and the edges $\{i,j\}$  are denoted $ij$.

\medskip


\tiny
\setlist[itemize]{nosep}
\begin{enumerate}[nosep]
\item  $\{01, 02, 03, 04, 12, 13, 14, 23, 24, 34\}$
\item  $\{02, 03, 05, 06, 13, 14, 15, 16, 24, 25, 26, 35, 36, 45, 46, 56\}$
\item  $\{02, 04, 05, 07, 13, 15, 16, 17, 24, 26, 27, 35, 36, 37, 46, 47, 57, 67\}$
\item  $\{02, 04, 05, 06, 07, 13, 15, 16, 17, 24, 25, 27, 35, 36, 37, 46, 47, 57, 67\}$
\item  $\{02, 04, 05, 06, 07, 13, 14, 16, 17, 24, 25, 27, 35, 36, 37, 46, 47, 57, 67\}$
\item  $\{02, 04, 05, 06, 07, 13, 14, 15, 16, 17, 24, 26, 27, 35, 36, 37, 45, 47, 57, 67\}$
\item  $\{02, 03, 05, 06, 07, 13, 14, 15, 16, 17, 24, 25, 27, 35, 36, 37, 46, 47, 57, 67\}$
\item  $\{02, 03, 04, 05, 07, 13, 14, 15, 16, 17, 24, 25, 26, 27, 35, 36, 37, 46, 47, 57, 67\}$
\item  $\{02, 04, 06, 07, 13, 15, 17, 18, 24, 26, 28, 35, 37, 38, 46, 48, 57, 58, 68\}$
\item  $\{02, 04, 06, 07, 13, 15, 17, 18, 24, 26, 27, 35, 37, 38, 46, 48, 57, 58, 68\}$
\item  $\{02, 04, 06, 07, 08, 13, 15, 17, 18, 24, 26, 27, 35, 37, 38, 46, 48, 57, 58, 68\}$
\item  $\{02, 04, 06, 07, 08, 13, 15, 17, 18, 24, 26, 27, 28, 35, 37, 38, 46, 47, 57, 58, 68\}$
\item  $\{02, 04, 06, 07, 08, 13, 15, 16, 18, 24, 26, 27, 28, 35, 37, 38, 46, 47, 57, 58, 68\}$
\item  $\{02, 04, 06, 07, 13, 15, 16, 17, 18, 24, 26, 28, 35, 37, 38, 46, 48, 57, 58, 68\}$
\item  $\{02, 04, 06, 07, 13, 15, 16, 17, 18, 24, 26, 28, 35, 37, 38, 46, 48, 57, 58, 67, 68\}$
\item  $\{02, 04, 06, 07, 08, 13, 15, 16, 17, 18, 24, 26, 28, 35, 37, 38, 46, 48, 57, 58, 67\}$
\item  $\{02, 04, 06, 07, 08, 13, 15, 16, 17, 24, 26, 27, 28, 35, 37, 38, 46, 48, 57, 58, 78\}$
\item  $\{02, 04, 06, 07, 08, 13, 15, 16, 17, 18, 24, 26, 27, 35, 37, 38, 46, 48, 57, 58, 68\}$
\item  $\{02, 04, 06, 07, 08, 13, 15, 16, 17, 18, 24, 26, 27, 28, 35, 37, 38, 46, 48, 57, 58, 67\}$
\item  $\{02, 04, 06, 07, 08, 13, 15, 16, 17, 18, 24, 26, 27, 35, 36, 38, 46, 47, 57, 58, 68, 78\}$
\item  $\{02, 04, 06, 07, 08, 13, 15, 16, 17, 18, 24, 26, 27, 35, 36, 37, 38, 46, 48, 57, 58, 68\}$
\item  $\{02, 04, 06, 07, 08, 13, 15, 16, 17, 18, 24, 26, 27, 28, 35, 36, 37, 38, 46, 47, 57, 58, 68\}$
\item  $\{02, 04, 05, 06, 07, 13, 16, 17, 18, 24, 25, 26, 27, 36, 37, 38, 45, 46, 48, 57, 58, 68, 78\}$
\item  $\{02, 04, 05, 08, 13, 15, 16, 17, 18, 24, 26, 27, 28, 35, 36, 37, 38, 46, 47, 58, 67, 68\}$
\item  $\{02, 04, 05, 08, 13, 15, 16, 17, 18, 24, 26, 27, 28, 35, 36, 37, 38, 46, 47, 48, 58, 67\}$
\item  $\{02, 04, 05, 07, 13, 15, 16, 17, 18, 24, 26, 28, 35, 36, 37, 38, 46, 48, 57, 58, 67, 68\}$
\item  $\{02, 04, 05, 07, 08, 13, 15, 16, 17, 24, 26, 28, 35, 36, 37, 46, 48, 57, 58, 67, 68, 78\}$
\item  $\{02, 04, 05, 07, 08, 13, 15, 16, 17, 18, 24, 26, 28, 35, 36, 37, 46, 48, 57, 58, 67, 78\}$
\item  $\{02, 04, 05, 07, 08, 13, 15, 16, 17, 18, 24, 26, 28, 35, 36, 37, 38, 46, 48, 57, 58, 67\}$
\item  $\{02, 04, 05, 07, 13, 15, 16, 17, 18, 24, 26, 27, 28, 35, 36, 38, 46, 47, 58, 67, 68, 78\}$
\item  $\{02, 04, 05, 07, 13, 15, 16, 17, 18, 24, 26, 27, 28, 35, 36, 38, 46, 47, 48, 58, 67, 68\}$
\item  $\{02, 04, 05, 07, 08, 13, 15, 16, 17, 24, 26, 27, 28, 35, 36, 38, 46, 47, 48, 58, 67, 68\}$
\item  $\{02, 04, 05, 07, 08, 13, 15, 16, 17, 18, 24, 26, 27, 35, 36, 38, 46, 47, 58, 67, 68, 78\}$
\item  $\{02, 04, 05, 07, 08, 13, 15, 16, 17, 18, 24, 26, 27, 28, 35, 36, 38, 46, 47, 58, 67, 78\}$
\item  $\{02, 04, 05, 07, 08, 13, 15, 16, 17, 18, 24, 26, 27, 28, 35, 36, 38, 46, 47, 48, 58, 67\}$
\item  $\{02, 04, 05, 07, 13, 15, 16, 17, 18, 24, 26, 27, 28, 35, 36, 37, 38, 46, 48, 57, 68\}$
\item  $\{02, 04, 05, 07, 13, 15, 16, 17, 18, 24, 26, 27, 28, 35, 36, 37, 38, 46, 48, 57, 58, 68\}$
\item  $\{02, 04, 05, 07, 08, 13, 15, 16, 17, 18, 24, 26, 27, 35, 36, 37, 38, 46, 48, 57, 58, 68\}$
\item  $\{02, 04, 05, 07, 08, 13, 15, 16, 17, 18, 24, 26, 27, 28, 35, 36, 37, 46, 48, 57, 68\}$
\item  $\{02, 04, 05, 07, 08, 13, 15, 16, 17, 18, 24, 26, 27, 28, 35, 36, 37, 46, 48, 57, 58, 78\}$
\item  $\{02, 04, 05, 07, 08, 13, 15, 16, 17, 18, 24, 26, 27, 28, 35, 36, 37, 46, 48, 57, 58, 68\}$
\item  $\{02, 04, 05, 07, 08, 13, 15, 16, 17, 18, 24, 26, 27, 28, 35, 36, 37, 38, 46, 48, 57, 68\}$
\item  $\{02, 04, 05, 07, 08, 13, 15, 16, 17, 18, 24, 26, 27, 28, 35, 36, 37, 38, 46, 48, 57, 58, 67\}$
\item  $\{02, 04, 05, 07, 08, 13, 15, 16, 17, 18, 24, 26, 27, 28, 35, 36, 37, 38, 46, 47, 48, 58, 67\}$
\item  $\{02, 04, 05, 06, 08, 13, 15, 16, 17, 24, 27, 28, 35, 36, 37, 47, 48, 56, 58, 68\}$
\item  $\{02, 04, 05, 06, 08, 13, 15, 16, 17, 18, 24, 27, 28, 35, 36, 37, 47, 48, 56, 78\}$
\item  $\{02, 04, 05, 06, 08, 13, 15, 16, 17, 18, 24, 27, 28, 35, 36, 37, 47, 48, 56, 58, 78\}$
\item  $\{02, 04, 05, 06, 08, 13, 15, 16, 17, 18, 24, 27, 28, 35, 36, 37, 47, 48, 56, 58, 68\}$
\item  $\{02, 04, 05, 06, 08, 13, 15, 16, 17, 18, 24, 27, 28, 35, 36, 37, 38, 47, 48, 56, 78\}$
\item  $\{02, 04, 05, 06, 07, 08, 13, 15, 16, 17, 18, 24, 27, 28, 35, 36, 37, 47, 48, 56, 58, 68\}$
\item  $\{02, 04, 05, 06, 08, 13, 15, 16, 17, 18, 24, 26, 27, 35, 37, 38, 46, 48, 57, 58, 68\}$
\item  $\{02, 04, 05, 06, 08, 13, 15, 16, 17, 18, 24, 26, 27, 35, 37, 38, 46, 47, 57, 58, 68\}$
\item  $\{02, 04, 05, 06, 08, 13, 15, 16, 17, 18, 24, 26, 27, 28, 35, 37, 38, 46, 47, 57, 58, 68\}$
\item  $\{02, 04, 05, 06, 07, 08, 13, 15, 16, 17, 18, 24, 26, 28, 35, 37, 38, 46, 48, 57, 58, 67\}$
\item  $\{02, 04, 05, 06, 07, 13, 15, 16, 17, 18, 24, 26, 27, 28, 35, 37, 38, 46, 48, 57, 58, 68\}$
\item  $\{02, 04, 05, 06, 07, 08, 13, 15, 16, 17, 18, 24, 26, 27, 35, 37, 38, 46, 48, 57, 58, 68\}$
\item  $\{02, 04, 05, 06, 07, 13, 15, 16, 17, 18, 24, 26, 27, 28, 35, 37, 38, 46, 47, 48, 57, 58, 68\}$
\item  $\{02, 04, 05, 06, 13, 15, 16, 17, 18, 24, 26, 27, 28, 35, 36, 37, 38, 47, 48, 56, 57, 68, 78\}$
\item  $\{02, 04, 05, 06, 08, 13, 15, 16, 17, 24, 26, 27, 28, 35, 36, 37, 47, 48, 56, 57, 58, 68, 78\}$
\item  $\{02, 04, 05, 06, 08, 13, 15, 16, 17, 18, 24, 26, 27, 28, 35, 36, 37, 47, 48, 56, 57, 68, 78\}$
\item  $\{02, 04, 05, 06, 08, 13, 15, 16, 17, 18, 24, 26, 27, 28, 35, 36, 37, 47, 48, 56, 57, 58, 78\}$
\item  $\{02, 04, 05, 06, 08, 13, 15, 16, 17, 18, 24, 26, 27, 28, 35, 36, 37, 47, 48, 56, 57, 58, 68\}$
\item  $\{02, 04, 05, 06, 08, 13, 15, 16, 17, 18, 24, 26, 27, 28, 35, 36, 37, 38, 47, 48, 56, 57, 68\}$
\item  $\{02, 04, 05, 06, 07, 13, 15, 16, 17, 18, 24, 26, 28, 35, 36, 37, 38, 47, 48, 56, 57, 68, 78\}$
\item  $\{02, 04, 05, 06, 07, 08, 13, 15, 16, 17, 18, 24, 26, 28, 35, 36, 37, 47, 48, 56, 57, 68, 78\}$
\item  $\{02, 04, 05, 06, 07, 13, 15, 16, 17, 18, 24, 26, 27, 28, 35, 36, 38, 47, 48, 56, 57, 68, 78\}$
\item  $\{02, 04, 05, 06, 07, 08, 13, 15, 16, 17, 24, 26, 27, 28, 35, 36, 38, 47, 48, 56, 57, 58, 68\}$
\item  $\{02, 04, 05, 06, 13, 15, 16, 17, 18, 24, 26, 27, 28, 35, 36, 37, 38, 46, 47, 48, 57, 68, 78\}$
\item  $\{02, 04, 05, 06, 07, 13, 15, 16, 17, 18, 24, 26, 27, 28, 35, 36, 37, 38, 46, 48, 57, 68, 78\}$
\item  $\{02, 04, 05, 06, 07, 13, 15, 16, 17, 18, 24, 26, 27, 28, 35, 36, 37, 38, 46, 48, 57, 58, 68\}$
\item  $\{02, 04, 05, 06, 07, 08, 13, 15, 16, 17, 18, 24, 26, 27, 35, 36, 37, 38, 46, 48, 57, 58, 68\}$
\item  $\{02, 04, 05, 06, 07, 13, 15, 16, 17, 18, 24, 26, 27, 28, 35, 36, 37, 38, 46, 47, 48, 57, 58, 68\}$
\item  $\{02, 04, 05, 06, 07, 13, 15, 16, 17, 18, 24, 25, 26, 28, 36, 37, 38, 45, 47, 48, 57, 68, 78\}$
\item  $\{02, 04, 05, 06, 07, 08, 13, 15, 16, 17, 18, 24, 25, 26, 28, 36, 37, 38, 45, 47, 57, 68, 78\}$
\item  $\{02, 04, 05, 06, 07, 08, 13, 15, 16, 17, 18, 24, 25, 26, 28, 36, 37, 38, 45, 47, 57, 58, 68\}$
\item  $\{02, 04, 05, 06, 08, 13, 15, 16, 17, 24, 25, 27, 28, 35, 36, 37, 46, 47, 48, 57, 58, 68, 78\}$
\item  $\{02, 04, 05, 06, 08, 13, 15, 16, 17, 18, 24, 25, 27, 35, 36, 37, 38, 46, 47, 48, 57, 68\}$
\item  $\{02, 04, 05, 06, 08, 13, 15, 16, 17, 18, 24, 25, 27, 35, 36, 37, 38, 46, 47, 48, 57, 58, 68\}$
\item  $\{02, 04, 05, 06, 08, 13, 15, 16, 17, 18, 24, 25, 27, 28, 35, 36, 37, 46, 47, 48, 57, 58, 78\}$
\item  $\{02, 04, 05, 06, 08, 13, 15, 16, 17, 18, 24, 25, 27, 28, 35, 36, 37, 46, 47, 48, 57, 58, 68\}$
\item  $\{02, 04, 05, 06, 08, 13, 15, 16, 17, 18, 24, 25, 27, 28, 35, 36, 37, 38, 46, 47, 57, 58, 68\}$
\item  $\{02, 04, 05, 06, 08, 13, 15, 16, 17, 18, 24, 25, 27, 28, 35, 36, 37, 38, 46, 47, 48, 57, 68\}$
\item  $\{02, 04, 05, 06, 08, 13, 15, 16, 17, 18, 24, 25, 27, 28, 35, 36, 37, 46, 47, 48, 57, 58, 67, 68\}$
\item  $\{02, 04, 05, 06, 08, 13, 15, 16, 17, 18, 24, 25, 27, 28, 35, 36, 37, 38, 46, 47, 57, 58, 67, 68\}$
\item  $\{02, 04, 05, 06, 07, 08, 13, 15, 16, 17, 18, 24, 25, 27, 35, 36, 37, 38, 46, 48, 57, 58, 68\}$
\item  $\{02, 04, 05, 06, 07, 08, 13, 15, 16, 17, 18, 24, 25, 27, 28, 35, 36, 37, 38, 46, 48, 57, 68\}$
\item  $\{02, 04, 05, 06, 07, 13, 15, 16, 17, 18, 24, 25, 27, 28, 35, 36, 37, 38, 46, 47, 48, 56, 58, 78\}$
\item  $\{02, 04, 05, 06, 07, 08, 13, 15, 16, 17, 18, 24, 25, 27, 28, 35, 36, 37, 38, 46, 47, 56, 58, 78\}$
\item  $\{02, 04, 05, 06, 07, 13, 15, 16, 17, 18, 24, 25, 26, 27, 35, 37, 38, 46, 47, 48, 56, 58, 68, 78\}$
\item  $\{02, 04, 05, 06, 07, 08, 13, 15, 16, 17, 18, 24, 25, 26, 27, 35, 37, 38, 46, 47, 56, 58, 68, 78\}$
\item  $\{02, 04, 05, 06, 07, 08, 13, 15, 16, 17, 18, 24, 25, 26, 27, 35, 37, 38, 46, 47, 48, 56, 68, 78\}$
\item  $\{02, 04, 05, 06, 07, 08, 13, 15, 16, 17, 18, 24, 25, 26, 27, 35, 37, 38, 46, 47, 48, 56, 58, 68\}$
\item  $\{02, 04, 05, 06, 07, 08, 13, 15, 16, 17, 18, 24, 25, 26, 27, 35, 36, 37, 38, 47, 48, 56, 58, 78\}$
\item  $\{02, 04, 05, 06, 08, 13, 14, 16, 17, 18, 24, 25, 27, 35, 36, 37, 38, 46, 48, 57, 58, 68\}$
\item  $\{02, 04, 05, 06, 08, 13, 14, 16, 17, 18, 24, 25, 27, 28, 35, 36, 37, 46, 48, 57, 68\}$
\item  $\{02, 04, 05, 06, 08, 13, 14, 16, 17, 18, 24, 25, 27, 28, 35, 36, 37, 38, 46, 48, 57, 68\}$
\item  $\{02, 04, 05, 06, 08, 13, 14, 16, 17, 24, 25, 27, 28, 35, 36, 37, 38, 46, 47, 48, 57, 58, 78\}$
\item  $\{02, 04, 05, 06, 08, 13, 14, 16, 17, 18, 24, 25, 27, 35, 36, 37, 38, 46, 47, 57, 58, 68\}$
\item  $\{02, 04, 05, 06, 08, 13, 14, 16, 17, 18, 24, 25, 27, 35, 36, 37, 38, 46, 47, 48, 57, 58, 68\}$
\item  $\{02, 04, 05, 06, 08, 13, 14, 16, 17, 18, 24, 25, 27, 28, 35, 36, 37, 46, 47, 48, 57, 58, 78\}$
\item  $\{02, 04, 05, 06, 08, 13, 14, 16, 17, 18, 24, 25, 27, 28, 35, 36, 37, 38, 46, 47, 57, 58, 68\}$
\item  $\{02, 04, 05, 06, 08, 13, 14, 16, 17, 18, 24, 25, 27, 28, 35, 36, 37, 38, 46, 47, 48, 57, 68\}$
\item  $\{02, 04, 05, 06, 07, 08, 13, 14, 16, 17, 18, 24, 25, 28, 35, 36, 37, 46, 48, 57, 67, 68, 78\}$
\item  $\{02, 04, 05, 06, 07, 08, 13, 14, 16, 17, 18, 24, 25, 28, 35, 36, 37, 38, 46, 48, 57, 58, 67\}$
\item  $\{02, 04, 05, 06, 07, 13, 14, 16, 17, 18, 24, 25, 27, 28, 35, 36, 37, 38, 46, 48, 57, 58, 68\}$
\item  $\{02, 04, 05, 06, 07, 08, 13, 14, 16, 17, 18, 24, 25, 27, 35, 36, 37, 38, 46, 48, 57, 68, 78\}$
\item  $\{02, 04, 05, 06, 07, 08, 13, 14, 16, 17, 18, 24, 25, 27, 35, 36, 37, 38, 46, 48, 57, 58, 68\}$
\item  $\{02, 04, 05, 06, 07, 08, 13, 14, 16, 17, 18, 24, 25, 27, 28, 35, 36, 37, 46, 48, 57, 68, 78\}$
\item  $\{02, 04, 05, 06, 07, 08, 13, 14, 16, 17, 18, 24, 25, 27, 28, 35, 36, 37, 46, 48, 57, 58, 68\}$
\item  $\{02, 04, 05, 06, 07, 08, 13, 14, 16, 17, 18, 24, 25, 27, 28, 35, 36, 37, 38, 46, 48, 57, 68\}$
\item  $\{02, 04, 05, 06, 08, 13, 14, 16, 17, 24, 25, 27, 28, 35, 36, 37, 46, 47, 48, 56, 57, 58, 68, 78\}$
\item  $\{02, 04, 05, 06, 08, 13, 14, 16, 17, 18, 24, 25, 27, 28, 35, 36, 37, 38, 46, 47, 48, 56, 57, 68\}$
\item  $\{02, 04, 05, 06, 07, 13, 14, 16, 17, 18, 24, 25, 27, 28, 35, 36, 37, 38, 46, 47, 48, 56, 57, 58, 68\}$
\item  $\{02, 04, 05, 06, 07, 13, 14, 16, 18, 24, 25, 26, 27, 35, 37, 38, 46, 47, 48, 56, 57, 58, 68, 78\}$
\item  $\{02, 04, 05, 06, 07, 13, 14, 16, 17, 18, 24, 25, 26, 35, 37, 38, 46, 47, 48, 56, 57, 58, 68, 78\}$
\item  $\{02, 04, 05, 06, 07, 13, 14, 16, 17, 18, 24, 25, 26, 28, 35, 37, 38, 46, 47, 48, 56, 57, 58, 78\}$
\item  $\{02, 04, 05, 06, 07, 13, 14, 16, 17, 18, 24, 25, 26, 28, 35, 37, 38, 46, 47, 48, 56, 57, 58, 68\}$
\item  $\{02, 04, 05, 06, 07, 08, 13, 14, 16, 17, 18, 24, 25, 26, 35, 37, 38, 46, 47, 56, 57, 58, 68, 78\}$
\item  $\{02, 04, 05, 06, 07, 08, 13, 14, 16, 17, 18, 24, 25, 26, 27, 35, 37, 38, 46, 56, 57, 58, 68, 78\}$
\item  $\{02, 04, 05, 06, 07, 08, 13, 14, 16, 17, 18, 24, 25, 26, 27, 35, 37, 38, 46, 48, 56, 57, 58, 78\}$
\item  $\{02, 04, 05, 06, 07, 13, 14, 16, 17, 18, 24, 25, 26, 28, 35, 36, 37, 38, 46, 47, 57, 58, 68, 78\}$
\item  $\{02, 04, 05, 06, 07, 13, 14, 16, 17, 18, 24, 25, 26, 28, 35, 36, 37, 38, 46, 47, 48, 57, 58, 78\}$
\item  $\{02, 04, 05, 06, 07, 08, 13, 14, 16, 17, 24, 25, 26, 28, 35, 36, 37, 38, 46, 47, 57, 58, 68, 78\}$
\item  $\{02, 04, 05, 06, 07, 08, 13, 14, 16, 17, 24, 25, 26, 28, 35, 36, 37, 38, 46, 47, 48, 57, 58, 78\}$
\item  $\{02, 04, 05, 06, 07, 08, 13, 14, 16, 17, 18, 24, 25, 26, 35, 36, 37, 38, 46, 47, 57, 58, 68, 78\}$
\item  $\{02, 04, 05, 06, 07, 08, 13, 14, 16, 17, 18, 24, 25, 26, 35, 36, 37, 38, 46, 47, 48, 57, 68, 78\}$
\item  $\{02, 04, 05, 06, 07, 08, 13, 14, 16, 17, 18, 24, 25, 26, 28, 35, 36, 37, 38, 46, 47, 57, 58, 68\}$
\item  $\{02, 04, 05, 06, 07, 08, 13, 14, 16, 17, 18, 24, 25, 26, 27, 35, 36, 38, 46, 47, 48, 57, 68, 78\}$
\item  $\{02, 04, 05, 06, 07, 13, 14, 16, 17, 18, 24, 25, 26, 27, 35, 36, 37, 38, 46, 48, 57, 58, 68, 78\}$
\item  $\{02, 04, 05, 06, 07, 08, 13, 14, 16, 17, 18, 24, 25, 26, 27, 35, 36, 37, 38, 46, 48, 57, 68, 78\}$
\item  $\{02, 04, 05, 06, 07, 08, 13, 14, 16, 17, 18, 24, 25, 26, 27, 35, 36, 37, 38, 46, 48, 57, 58, 68\}$
\item  $\{02, 04, 05, 06, 07, 13, 14, 16, 17, 18, 24, 25, 26, 28, 35, 36, 37, 38, 45, 47, 48, 57, 58, 78\}$
\item  $\{02, 04, 05, 06, 07, 08, 13, 14, 16, 17, 18, 24, 25, 26, 35, 36, 37, 38, 45, 47, 48, 57, 68, 78\}$
\item  $\{02, 04, 05, 06, 07, 08, 13, 14, 16, 17, 18, 24, 25, 26, 28, 35, 36, 37, 38, 45, 47, 57, 68, 78\}$
\item  $\{02, 04, 05, 06, 07, 13, 14, 16, 17, 18, 24, 25, 26, 27, 35, 36, 37, 38, 45, 46, 48, 57, 58, 68, 78\}$
\item  $\{02, 04, 05, 06, 07, 08, 13, 14, 16, 17, 18, 24, 25, 26, 27, 35, 36, 37, 38, 45, 46, 48, 57, 68, 78\}$
\item  $\{02, 04, 05, 06, 07, 08, 13, 14, 16, 17, 18, 24, 25, 26, 27, 35, 36, 37, 38, 45, 46, 48, 57, 58, 68\}$
\item  $\{02, 04, 05, 06, 08, 13, 14, 15, 17, 18, 24, 26, 27, 35, 36, 37, 38, 46, 48, 57, 58, 68\}$
\item  $\{02, 04, 05, 06, 07, 08, 13, 14, 15, 17, 18, 24, 26, 28, 35, 36, 37, 46, 47, 48, 57, 58, 78\}$
\item  $\{02, 04, 05, 06, 07, 08, 13, 14, 15, 17, 18, 24, 26, 27, 35, 36, 38, 46, 47, 57, 58, 68, 78\}$
\item  $\{02, 04, 05, 06, 07, 08, 13, 14, 15, 17, 18, 24, 26, 27, 35, 36, 38, 46, 47, 48, 57, 58, 68\}$
\item  $\{02, 04, 05, 06, 07, 13, 14, 15, 17, 18, 24, 26, 27, 28, 35, 36, 37, 38, 46, 48, 57, 68\}$
\item  $\{02, 04, 05, 06, 07, 13, 14, 15, 17, 18, 24, 26, 27, 28, 35, 36, 37, 38, 46, 48, 57, 58, 68\}$
\item  $\{02, 04, 05, 06, 07, 08, 13, 14, 15, 17, 18, 24, 26, 27, 35, 36, 37, 38, 46, 48, 57, 58, 68\}$
\item  $\{02, 04, 05, 06, 07, 08, 13, 14, 15, 17, 18, 24, 26, 27, 28, 35, 36, 37, 46, 48, 57, 58, 78\}$
\item  $\{02, 04, 05, 06, 07, 08, 13, 14, 15, 16, 17, 18, 24, 26, 27, 35, 37, 38, 46, 47, 48, 56, 57, 58, 68\}$
\item  $\{02, 04, 05, 06, 08, 13, 14, 15, 16, 17, 24, 26, 27, 28, 35, 36, 37, 38, 46, 47, 57, 58, 68, 78\}$
\item  $\{02, 04, 05, 06, 07, 13, 14, 15, 16, 18, 24, 26, 27, 28, 35, 36, 37, 38, 46, 47, 57, 58, 68, 78\}$
\item  $\{02, 04, 05, 06, 07, 13, 14, 15, 16, 18, 24, 26, 27, 28, 35, 36, 37, 38, 46, 47, 48, 57, 68, 78\}$
\item  $\{02, 04, 05, 06, 07, 08, 13, 14, 15, 16, 18, 24, 26, 27, 35, 36, 37, 38, 46, 47, 57, 58, 68, 78\}$
\item  $\{02, 04, 05, 06, 07, 08, 13, 14, 15, 16, 18, 24, 26, 27, 28, 35, 36, 37, 38, 46, 47, 57, 58, 68\}$
\item  $\{02, 04, 05, 06, 07, 08, 13, 14, 15, 16, 17, 18, 24, 26, 28, 35, 36, 37, 38, 46, 47, 57, 58, 68\}$
\item  $\{02, 04, 05, 06, 07, 13, 14, 15, 16, 17, 18, 24, 26, 27, 28, 35, 36, 37, 38, 46, 48, 57, 58, 68\}$
\item  $\{02, 04, 05, 06, 07, 08, 13, 14, 15, 16, 17, 18, 24, 26, 27, 35, 36, 37, 38, 46, 48, 57, 58, 68\}$
\item  $\{02, 04, 05, 06, 07, 08, 13, 14, 15, 16, 17, 18, 24, 26, 27, 28, 35, 36, 37, 38, 46, 47, 57, 58, 68\}$
\item  $\{02, 04, 05, 06, 07, 08, 13, 14, 15, 16, 17, 18, 24, 26, 27, 35, 36, 37, 38, 45, 47, 48, 58, 67, 68\}$
\item  $\{02, 04, 05, 06, 07, 08, 13, 14, 15, 16, 17, 18, 24, 26, 27, 28, 35, 36, 37, 38, 45, 46, 57, 68, 78\}$
\item  $\{02, 04, 05, 06, 07, 08, 13, 14, 15, 16, 17, 18, 24, 25, 26, 27, 35, 36, 37, 38, 47, 48, 56, 58, 78\}$
\item  $\{02, 03, 05, 06, 08, 13, 14, 15, 16, 17, 18, 24, 25, 27, 28, 35, 36, 37, 46, 47, 48, 57, 58, 68\}$
\item  $\{02, 03, 05, 06, 08, 13, 14, 15, 16, 17, 18, 24, 25, 27, 28, 35, 36, 37, 38, 46, 47, 57, 58, 68\}$
\item  $\{02, 03, 05, 06, 07, 08, 13, 14, 15, 16, 17, 18, 24, 25, 27, 28, 35, 36, 37, 46, 48, 57, 58, 67, 68\}$
\item  $\{02, 03, 05, 06, 07, 13, 14, 15, 16, 17, 18, 24, 25, 27, 28, 35, 36, 37, 38, 46, 47, 48, 56, 58, 78\}$
\item  $\{02, 03, 05, 06, 07, 08, 13, 14, 15, 16, 17, 24, 25, 27, 28, 35, 36, 37, 46, 47, 48, 56, 58, 68, 78\}$
\item  $\{02, 03, 05, 06, 07, 08, 13, 14, 15, 16, 17, 18, 24, 25, 27, 28, 35, 36, 37, 46, 47, 48, 56, 58, 78\}$
\item  $\{02, 03, 05, 06, 07, 08, 13, 14, 15, 16, 17, 18, 24, 25, 27, 28, 35, 36, 37, 46, 47, 48, 56, 58, 68\}$
\item  $\{02, 03, 05, 06, 07, 08, 13, 14, 15, 16, 17, 18, 24, 25, 27, 28, 35, 36, 37, 38, 46, 47, 56, 58, 78\}$
\item  $\{02, 03, 05, 06, 07, 13, 14, 15, 16, 17, 18, 24, 25, 26, 27, 28, 35, 37, 38, 46, 47, 48, 56, 58, 78\}$
\item  $\{02, 03, 05, 06, 07, 08, 13, 14, 15, 16, 17, 18, 24, 25, 26, 27, 35, 37, 38, 46, 47, 48, 56, 58, 78\}$
\item  $\{02, 03, 05, 06, 07, 08, 13, 14, 15, 16, 17, 18, 24, 25, 26, 27, 28, 35, 37, 38, 46, 47, 56, 58, 78\}$
\item  $\{02, 03, 05, 06, 07, 08, 13, 14, 15, 16, 17, 18, 24, 25, 26, 27, 35, 36, 37, 38, 46, 48, 57, 58, 68\}$
\item  $\{02, 03, 05, 06, 07, 08, 13, 14, 15, 16, 17, 18, 24, 25, 26, 27, 28, 35, 36, 37, 46, 48, 57, 58, 68\}$
\item  $\{02, 03, 05, 06, 07, 08, 13, 14, 15, 16, 17, 18, 24, 25, 26, 27, 28, 35, 36, 37, 38, 46, 47, 57, 58, 68\}$
\item  $\{02, 03, 04, 06, 07, 08, 13, 14, 15, 16, 17, 18, 24, 25, 26, 35, 36, 37, 38, 46, 47, 48, 57, 68, 78\}$
\item  $\{02, 03, 04, 06, 07, 08, 13, 14, 15, 16, 17, 18, 24, 25, 26, 27, 35, 36, 37, 38, 46, 47, 48, 57, 58, 68\}$
\item  $\{02, 03, 04, 06, 07, 08, 13, 14, 15, 16, 17, 18, 24, 25, 26, 27, 28, 35, 36, 37, 38, 46, 47, 57, 58, 68\}$
\item  $\{02, 03, 04, 05, 07, 08, 13, 14, 15, 16, 17, 18, 24, 25, 26, 27, 28, 35, 36, 37, 38, 46, 47, 48, 57, 68\}$
\item  $\{02, 03, 04, 05, 06, 08, 13, 14, 15, 16, 17, 18, 24, 25, 26, 27, 28, 35, 36, 37, 38, 46, 47, 57, 58, 68\}$
\item  $\{02, 03, 04, 05, 06, 07, 13, 14, 15, 16, 17, 18, 24, 25, 26, 27, 28, 35, 36, 37, 38, 46, 47, 48, 57, 58, 68\}$
\end{enumerate}

\end{document}